\documentclass[11pt,reqno]{amsart}
\usepackage[comma]{natbib}
\usepackage[a4paper,left=1in,right=1in,top=1.25in,bottom=1.25in]{geometry}
\usepackage{color}
\usepackage{graphicx}
\usepackage{amsmath}
\usepackage{amssymb}
\usepackage{enumerate}
\usepackage{comment}

\newtheorem{theorem}{Theorem}%[section]
\newtheorem{corollary}[theorem]{Corollary}

\theoremstyle{definition}
\newtheorem{definition}{Definition}
\theoremstyle{remark}

\theoremstyle{remark}
\newtheorem{example}{Example}

\begin{document}

\title[]{Lorenz and convex ordering of parasite burden distributions with density-dependent deaths}
\author[]{Madison Carlton}
\author[]{Ross McVinish}
\address{School of Mathematics and Physics, University of Queensland}
\email{r.mcvinish@uq.edu.au}

\subjclass[2020]{}


\begin{abstract}
We establish a convex ordering result for count distributions arising from a parasite acquisition model with density-dependent deaths. If the ratio of two death-rate sequences is non-increasing, then the corresponding mean-parametrised distributions are ordered in the convex order. This yields, as a special case, the convex ordering of the mean-parametrised Conway-Maxwell-Poisson distribution in its dispersion parameter.
\end{abstract}

\maketitle

\section{Introduction}

We consider the following model of parasite acquisition where parasites are subject to density dependent deaths, previously studied by \citet{AG:82} and \citet{BP:00}, among others. At birth the host is parasite free. During its lifetime, new infections of the host occur as a homogeneous Poisson process with rate $\phi$. Acquired parasites die at a per capita mortality rate $\sigma(m)$, where $m$ is the number of parasites present in the host. We shall assume that parasite deaths exhibit positive density dependence, so 
\begin{equation}
0 < \sigma(m) \leq \sigma(m+1) < \infty,     \label{eq:parasite-death}
\end{equation}
for all $m \geq 1$. The host dies at a rate $\delta$, which does not depend on its parasite burden. 

Conditional on host survival, the distribution of the host's parasite burden converges as host age tends to infinity and the limiting probability mass function is given by
\begin{equation}
    \pi_m(\phi, \sigma) = \frac{1}{Z(\phi, \sigma)} \frac{\phi^m}{m! \prod_{k=1}^m \sigma(k)}, \quad m \in \{0,1,2,\ldots\}, \label{eq:stat-dist}
\end{equation}
where $Z(\phi,\sigma)$ is the normalising constant and the empty product is interpreted as one \citep{BP:00}. Two natural forms for $\sigma$ lead to known distributions.

\begin{example}
    Suppose $\sigma(m) = m^{\nu-1}$, with $\nu \geq 1$. Then
    \begin{equation*}
        \pi_m(\phi, \sigma) = \frac{1}{Z(\phi, \sigma)} \frac{\phi^m}{(m!)^\nu}.
    \end{equation*}
    This is the Conway-Maxwell-Poisson distribution \citep{CM:1961}. In general, any $\nu\geq 0$ is permissible. While much of the research on this distribution has focused on statistical modelling applications, \citet{DG:2016} and \citet{GX:2022} have examined some of its theoretical properties.
\end{example}
    
\begin{example}
    Suppose $\sigma(m) = \nu + m$, with $\nu > -1$. Then
    \begin{equation*}
        \pi_m(\phi, \sigma) = \frac{1}{Z(\phi,\sigma)} \frac{\phi^m}{m! \prod_{k=1}^m (k + \nu)} = \frac{1}{I_{\nu}(2\sqrt{\phi})} \frac{\phi^{m+\nu/2}}{m! \Gamma(m + \nu +1)},
    \end{equation*}
    where $I_\nu(\cdot)$ denotes the modified Bessel function of the first kind. This distribution is the discrete Bessel distribution introduced in \citet{YK:2000}, where in their notation $a = 2\sqrt{\phi}$. 
\end{example}    

\citet[Theorem 3]{BP:00} proved that the distribution (\ref{eq:stat-dist}) is under-dispersed, that is the variance-to-mean ratio is less than one for any non-constant $\sigma$ satisfying (\ref{eq:parasite-death}). While the variance-to-mean ratio has been a standard measure of aggregation in parasitology, \citet{Poulin:1993} has argued that having measures of aggregation based on the Lorenz curve is closer to what parasitologists intuitively think of when talking about aggregation. The Lorenz curve \citep{Lorenz:1905} was proposed as a graphical measure of wealth inequality. The following definition is due to \citet{Gastwirth:1971}.

\begin{definition}
    The {\em Lorenz curve} $L: [0,1] \to [0,1] $ for the distribution $F$ on $[0,\infty)$ with finite mean $ \mu$ is given by
    \begin{equation*}
        L(u) = \frac{\int_{0}^{u} F^{-1}(y)\, dy}{\mu}, 
    \end{equation*}
    where $ F^{-1}$ is the quantile function
    \begin{equation*}
        F^{-1}(y) = \inf \{ x : F(x) \geq y \} \quad \text{for } y \in (0,1).
    \end{equation*}
\end{definition}

When all hosts are infected with the same number of parasites, the Lorenz curve is given by $L(u) = u$ and is called the egalitarian line. Inequality in a distribution can be judged by how far its Lorenz curve departs from the egalitarian line. In this way, the Lorenz curve induces a partial ordering on the set of all distributions on $[0,\infty)$ with finite mean \citep[Definition 3.2.1]{Arnold:87}.

\begin{definition}
    Let $ X$ and $Y$ be random variables with the respective Lorenz curves denoted $ L_{X} $ and $ L_{Y}$. We say $ X$ is smaller than $Y$ in the {\em Lorenz order}, denoted $ X \leq_{\rm{Lorenz}} Y $ if $ L_{X}(u) \geq L_{Y}(u)$ for every $ u \in [0,1]$.
\end{definition} 

We do not work with the Lorenz curve directly, instead we exploit a connection with the convex order of distributions.

\begin{definition}
    For two random variables $X$ and $Y$ we say $X$ is smaller than $Y$ in the {\em convex order}, denoted $X \leq_{\rm cx} Y$, if $ \mathbb{E} \left[\psi(X)\right] \leq \mathbb{E} \left[\psi(Y)\right]$ for all convex functions $ \psi: \mathbb{R} \to \mathbb{R}$, provided the expectations exist. 
\end{definition}

The Lorenz order and convex order are related by the following:
\begin{align*}
X &\leq_{\rm Lorenz} Y    & \text{is equivalent to}&  &\frac{X}{\mathbb{E} \left[X\right]} &\leq_{\rm cx} \frac{Y}{\mathbb{E} \left[Y\right]}. 
\end{align*}
See \citet[Corollary 3.2.1]{Arnold:87} or \citet[Equation 3.A.33]{SS:07}. So when the two distributions being compared have the same mean, the Lorenz order reduces to the convex order. 

\section{Main result}

\subsection{Mean parametrisation} 

Although the distribution (\ref{eq:stat-dist}) is naturally parameterised by the infection rate $\phi$ and the death-rate sequence $\sigma$, it will be convenient to re-parametrise the distribution in terms of its mean $\mu$ and $\sigma$. Mean parametrisation are useful in statistical modelling applications \citep{Huang:2017}. Here, we use this parametrisation to disentangle the effect of the parasite death rates on mean abundance from their effect on aggregation.

Let $\mu(\phi,\sigma)$ denote the mean of the distribution (\ref{eq:stat-dist}). For the mean parametrisation to be well defined, we require that, for each death-rate sequence $\sigma$, the map
\begin{equation}
    \phi \mapsto \mu(\phi,\sigma) \label{eq:mean-map}
\end{equation}
is a one-to-one map onto $(0,\infty)$. In that case, for each mean $\mu \in (0,\infty)$ there is a unique infection rate $\phi(\mu,\sigma)$. For fixed $\sigma$, the distribution (\ref{eq:stat-dist}) is a power series distribution, and consequently $\mu(\phi,\sigma)$ is a strictly increasing function of $\phi$ \citep[Equations (2.4) and (2.5)]{JKK:05}. Hence the map (\ref{eq:mean-map}) is one-to-one. To show that (\ref{eq:mean-map}) is onto $(0,\infty)$ we need $\lim_{\phi\to 0} \mu(\phi, \sigma) = 0$ and $\lim_{\phi\to\infty}\mu(\phi,\sigma) = \infty$. As $Z(\phi,\sigma) \geq 1$ for any $\phi>0$ and $\sigma$, 
\[
\mu(\phi, \sigma) = \frac{1}{Z(\phi,\sigma)} \sum_{m=0}^\infty \frac{m\phi^m}{m! \prod_{k=1}^m \sigma(k)} \leq \phi \exp(\phi/\sigma(1))
\]
so $\lim_{\phi \to 0} \mu(\phi, \sigma) = 0$. To show $\lim_{\phi \to \infty} \mu(\phi,\sigma) = \infty$, let $n$ be the largest integer such that $n \sigma(n) < \phi/2$. Then
\begin{equation}
\frac{\sum_{k=0}^n \pi_k(\phi,\sigma)}{\pi_n(\phi,\sigma)} \leq \sum_{k=0}^n \left(\frac{n! \phi^k \prod_{j=1}^n \sigma(j)}{k! \phi^n \prod_{j=1}^k \sigma(j)} \right) \leq \sum_{k=0}^n \left(\frac{n\sigma(n)}{\phi} \right)^{n-k}  \leq 2. \label{eq:prob_bound}
\end{equation}
If $\pi_n(\phi,\sigma) > 1/3$, then $\mu(\phi,\sigma) > n/3$. On the other hand, if $\pi_n(\phi,\sigma) \leq 1/3$, then (\ref{eq:prob_bound}) implies 
\[
1 = \sum_{k=0}^\infty  \pi_k(\phi,\sigma) \leq 2\pi_n(\phi,\sigma) + \sum_{k=n+1}^\infty \pi_k(\phi,\sigma) \leq \frac{2}{3} + \sum_{k=n+1}^\infty \pi_k(\phi,\sigma)
\]
so $\mu(\phi,\sigma) > n/3$. Since $m\sigma(m)$ is increasing in $m$, it follows that $n\to\infty$ as $\phi\to\infty$ and $\lim_{\phi \to \infty} \mu(\phi,\sigma) = \infty$. This establishes the map (\ref{eq:mean-map}) is onto $(0, \infty)$ and the mean parameterisation is well defined. In the following we write $\pi(\mu,\sigma)$ for the distribution $\pi(\phi(\mu,\sigma), \sigma)$.

\subsection{Convex ordering in $\sigma$}

\begin{theorem} \label{thm:cx_order}
   If $\sigma_2(m)/\sigma_1(m)$ is non-increasing in $m$, then $\pi(\mu,\sigma_1) \leq_{\rm cx} \pi(\mu, \sigma_2)$.
\end{theorem}

\begin{proof}
    If $\sigma_2(m)/\sigma_1(m)$ is constant in $m$, then $\pi(\mu,\sigma_1) \stackrel{d}{=} \pi(\mu, \sigma_2)$ and the result holds trivially. Now suppose that  $\sigma_2(m)/\sigma_1(m)$ is not constant in $m$ so $\pi(\mu,\sigma_1)$ and $\pi(\mu, \sigma_2)$ are not identical distributions.  From \citet[Theorem 3.A.44]{SS:07}, $\pi(\mu,\sigma_1) \leq_{\rm cx} \pi(\mu, \sigma_2)$, if the difference in the pmfs $\pi_m(\mu, \sigma_2) - \pi_m(\mu,\sigma_1),\ m \geq 0$ has two sign changes with sign sequence $+, -, +$. Equivalently,  $\pi(\mu,\sigma_1) \leq_{\rm cx} \pi(\mu, \sigma_2)$, if
    \begin{equation}
    \frac{\pi_m(\mu, \sigma_2)}{\pi_m(\mu,\sigma_1)} - 1, \quad m \geq 0, \label{eq1:thm:cx1}
    \end{equation}
    has two sign changes with sign sequence $+,-,+$. The sequence (\ref{eq1:thm:cx1}) must have at least one sign change since $\pi_m(\mu,\sigma_1)$ and $\pi(\mu, \sigma_2)$ are pmfs. If there were only one sign change, then either $\pi(\mu,\sigma_1) \leq_{\rm st} \pi(\mu, \sigma_2)$ or $\pi(\mu, \sigma_2) \leq_{\rm st} \pi(\mu,\sigma_1)$, where $\leq_{\rm st}$ denotes the usual stochastic ordering, \citep[Theorem 1.A.12]{SS:07}. However, since the two distributions have the same mean, the usual stochastic ordering would imply the distributions were equal \citep[Theorem 1.A.8]{SS:07}. As the distributions are not equal, we must have at least two sign changes. 
    
    Now suppose $\sigma_2(m)/\sigma_1(m)$ is non-increasing in $m$ and not constant so 
    \begin{equation*}
    \frac{\sigma_2(m)}{\sigma_1(m)}    \frac{\sigma_1(m+1)}{\sigma_2(m+1)} \geq 1 
    \end{equation*}
    for all $m \geq 1$ and the inequality is strict for at least one $m$. Recall that a sequence $\{a_i\}_{i=0}^\infty$ is log-convex if $a_i^2 \leq a_{i-1} a_{i+1}$ for all $i\geq 1$. The sequence
    \[
    \frac{\pi_m(\mu, \sigma_2)}{\pi_m(\mu,\sigma_1)} = \frac{\pi_0(\mu, \sigma_2)}{\pi_0(\mu,\sigma_1)} \left(\frac{\phi(\mu,\sigma_2)}{\phi(\mu,\sigma_1)} \right)^m \prod_{k=1}^m \frac{\sigma_1(k)}{\sigma_2(k)}
    \]
    is log-convex since
    \begin{align*}
    \frac{\pi_{m-1}(\mu,\sigma_2)}{\pi_{m-1}(\mu,\sigma_1)} \frac{\pi_{m+1}(\mu, \sigma_2)}{\pi_{m+1}(\mu,\sigma_1)} & =   \left(\frac{\pi_0(\mu, \sigma_2)}{\pi_0(\mu,\sigma_1)}\right)^2 \left(\frac{\phi(\mu,\sigma_2)}{\phi(\mu,\sigma_1)} \right)^{2m} \left( \prod_{k=1}^{m-1} \frac{\sigma_1(k)}{\sigma_2(k)} \right)^2 \frac{\sigma_1(m)}{\sigma_2(m)} \frac{\sigma_1(m+1)}{\sigma_2(m+1)}\\
    & = \left(\frac{\pi_m(\mu, \sigma_2)}{\pi_m(\mu,\sigma_1)}\right)^2 \frac{\sigma_2(m)}{\sigma_1(m)}    \frac{\sigma_1(m+1)}{\sigma_2(m+1)} \geq \left(\frac{\pi_m(\mu,\sigma_2)}{\pi_m(\mu, \sigma_1)}\right)^2.
    \end{align*}
    A positive log-convex sequence is V-shaped, that is, it decreases to a minimum and is then increasing. 
    Hence, the sequence (\ref{eq1:thm:cx1}) has two sign changes and the sign sequence is $+,-,+$. 
\end{proof}

Taking $\sigma_2$ to be constant in Theorem \ref{thm:cx_order} immediately yields the following corollary.

\begin{corollary} \label{cor:cx_order}
    If $\sigma(m)$ is non-decreasing in $m$, then $\pi(\mu,\sigma) \leq_{\rm cx} \mathsf{Poisson}(\mu)$.
\end{corollary}

\begin{example}[Conway-Maxwell-Poisson distribution]
    Take $\sigma_1(m) = m^{\nu_1-1}$ and $\sigma_2(m) = m^{\nu_2-1}$ with $\nu_2<\nu_1$. Then 
    \[
        \frac{\sigma_2(m)}{\sigma_1(m)} = m^{\nu_2 - \nu_1}
    \]
    is decreasing in $m$. Writing $\mathsf{CMP}_\mu(\nu)$ for the mean parameterised Conway-Maxwell-Poisson distribution \citep{Huang:2017}, Theorem \ref{thm:cx_order} implies $\mathsf{CMP}_\mu(\nu)$ is decreasing in the convex order as $\nu$ increases and $\mu$ is fixed.
    %$\mathsf{CMP}_\mu(\nu_1) \leq_{\rm cx} \mathsf{CMP}_\mu(\nu_2)$. 
    As the CMP distribution reduces to the Poisson distribution when $\nu=1$, this strengthens the result of \citet[Proposition 2.12 (iii)]{DG:2016} who showed $\mathsf{CMP}_\mu(\nu) \leq_{\rm cx} \mathsf{Poisson}(\mu)$ for $\nu \geq 1$. 
\end{example}

\begin{example}[Bessel distribution]
    Take $\sigma_1(m) = \nu_1 + m$ and $\sigma_2(m) = \nu_2 + m$ with $-1<\nu_1 < \nu_2$. Then
    \[
    \frac{\sigma_2(m)}{\sigma_1(m)} = 1 + \frac{\nu_2 -\nu_1}{\nu_1+m}
    \]
    is decreasing in $m$. Writing $\mathsf{Bessel}_\mu(\nu)$ for the mean parameterised discrete Bessel distribution, Theorem \ref{thm:cx_order} implies $\mathsf{Bessel}_\mu(\nu)$ is increasing in the convex order as $\nu$ increases and $\mu$ is fixed.
    %$\mathsf{Bessel}_\mu(\nu_1) \leq_{\rm cx} \mathsf{Bessel}_\mu(\nu_2)$. 
    Corollary \ref{cor:cx_order} shows $\mathsf{Bessel}_\mu(\nu) \leq_{\rm cx} \mathsf{Poisson}(\mu)$ for all $\nu >-1$. 
\end{example}

\subsection{Role of infection rate and mean burden} We now ask whether, for a fixed death-rate sequence $\sigma$, the distributions $\pi(\phi,\sigma)$ can be ordered in the Lorenz sense as the infection rate $\phi$ varies. Since the map \eqref{eq:mean-map} is strictly increasing in $\phi$, the same question may be expressed in terms of the mean burden $\mu(\phi,\sigma)$. 

Several indices are monotone with respect to the Lorenz order, meaning if $X \leq_{\rm Lorenz} Y$, then the index $I$ satisfies $I(X) \leq I(Y)$. If we evaluate these indices  for $\pi(\phi,\sigma)$, then failure of any of them to be monotone in $\phi$ for fixed $\sigma$ would rule out the Lorenz ordering. One such index is the probability of zero parasite burden, $\mathbb{P}(X=0)$, which for $\pi(\phi,\sigma)$ is $Z(\phi,\sigma)^{-1}$. Since the normalising constant $Z(\phi,\sigma)$ is increasing in $\phi$, it follows that $\mathbb{P}(X=0)$ decreases as $\phi$ increases. Consequently, if the family $\{\pi(\phi,\sigma):\phi>0\}$, with $\sigma$ fixed, can be ordered in the Lorenz sense, then that ordering must be decreasing in $\phi$.

Two other indices that are monotone with respect to the Lorenz order are the Gini index \citep{Gini:2005} and Pietra index \citep{Pietra:1915}. The Gini index is defined as twice the area between the egalitarian line and the Lorenz curve.  \citet{Poulin:1993} has advocated the use of the Gini index in parasitology and it has become one of the standard measures of aggregation in that field. The Pietra index is defined as the maximal vertical deviation between the Lorenz curve and the egalitarian line. Alternative formulas for evaluating these indices are described in \citet[Sections 5.3.1. and 5.3.5]{Arnold:87}. 

Figure \ref{fig:CMP_contours_phi} shows contour plots of the Gini index and Pietra index calculated for $\sigma(m) = m^{\nu-1}$, which corresponds to the Conway-Maxwell-Poisson distribution. We see that both indices are non-monotone in $\phi$ for fixed $\nu$ sufficiently large, and non-monotone in $\nu$ for fixed $\phi$. In fact, the Pietra index appears to be increasing in $\nu$ for fixed $\phi$ sufficiently small.

A better understanding of the behaviour of these indices in this case (that is, the Conway-Maxwell-Poisson distribution), can be obtained from contour plots of the Gini index and Pietra index calculated using the mean parametrisation (Figure \ref{fig:CMP_contours_mu}). We see that these indices are decreasing in $\nu$ for fixed mean, as expected from Theorem \ref{thm:cx_order}, but they are not monotone decreasing in $\mu$ for sufficiently large $\nu$. This is particularly evident for the Pietra index, which takes smaller values when $\mu$ is close to an integer. This behaviour can be explained by a result of \citet{Huang:2023}.  As $\nu\to\infty$, $\pi(\mu,\sigma)$ converges to a distribution supported on $\lfloor \mu \rfloor$ and $\lceil \mu \rceil$ when $\mu$ is not an integer, and to the degenerate distribution at $\mu$ when $\mu$ is an integer. Therefore, the Gini and Pietra indices will be close to 0 for $\mu$ integer and $\nu$ sufficiently large, but bound away from 0 for non-integer $\mu$.

Figure \ref{fig:Bessel_contours} shows contour plots of the Gini index and Pietra index calculated for $\sigma(m) = \nu + m$, which corresponds to the mean parametrised Bessel distribution. We see that these indices are increasing in $\nu$ for fixed mean, as expected from Theorem \ref{thm:cx_order}, though the change is very slow particularly for means less than one. While less obvious than for the Conway-Maxwell-Poisson distribution, the Pietra index appears to be non-monotone in $\mu$ for $(\nu, \mu)$ near $(0,2)$ and  near $(0.5, 1)$.

\begin{figure}[h]
    \centering
    \includegraphics[width=\linewidth]{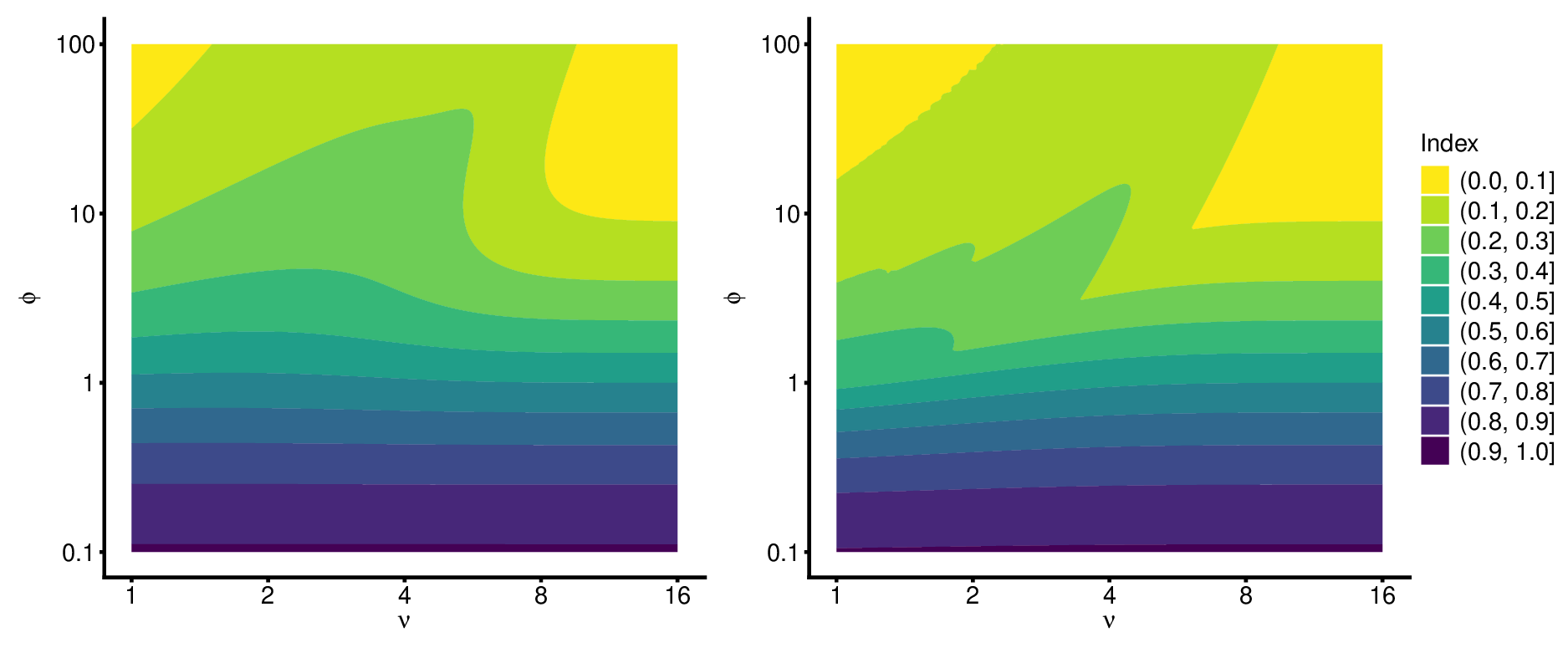}
    \caption{Contour plots of the Gini index (left) and Pietra index (right) for the Conway-Maxwell-Poisson distribution as functions of the infection rate, $\phi$ and mortality parameter $\nu$. Both axes are displayed on a logarithmic scale.}
    \label{fig:CMP_contours_phi}
\end{figure}

\begin{figure}[h]
    \centering
    \includegraphics[width=\linewidth]{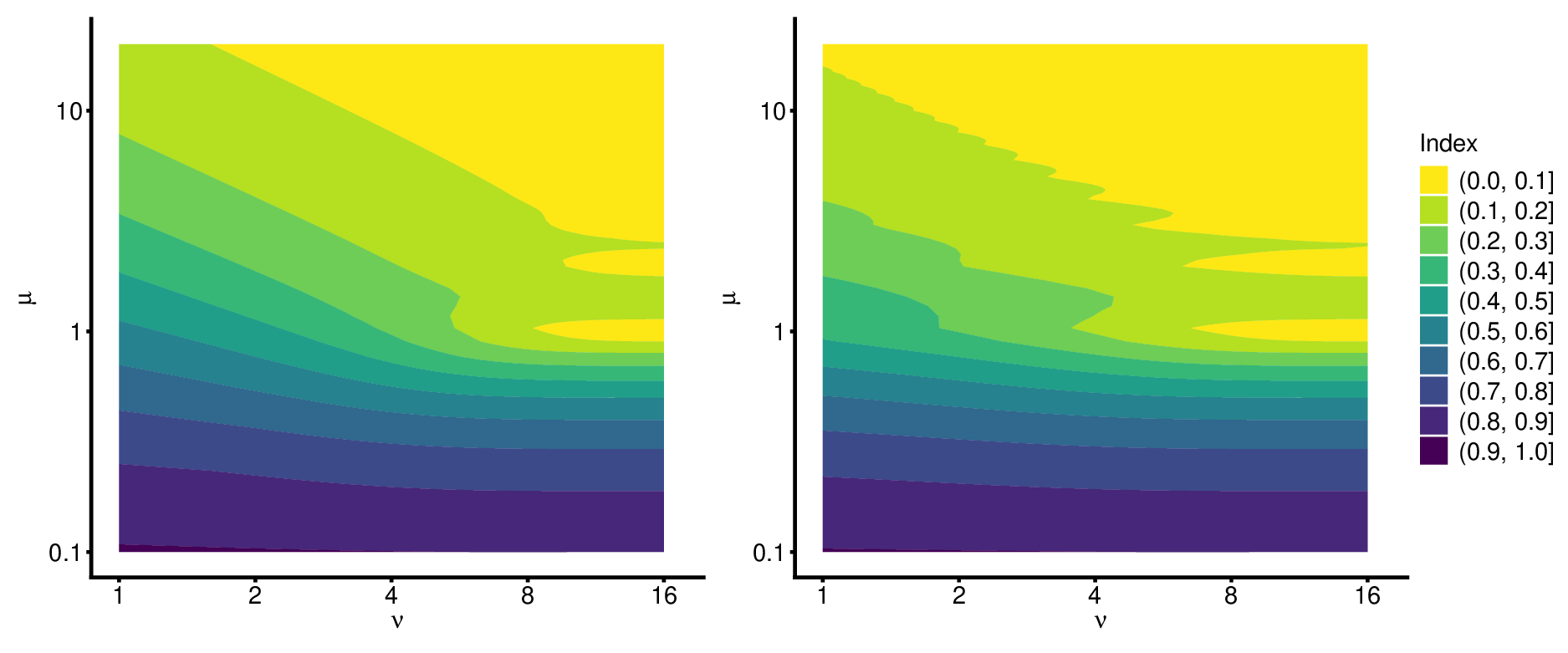}
    \caption{Contour plots of the Gini index (left) and Pietra index (right) for the Conway-Maxwell-Poisson distribution as functions of the infection rate, $\phi$ and mortality parameter $\nu$. Both axes are displayed on a logarithmic scale.}
    \label{fig:CMP_contours_mu}
\end{figure}

\begin{figure}[h]
    \centering
    \includegraphics[width=\linewidth]{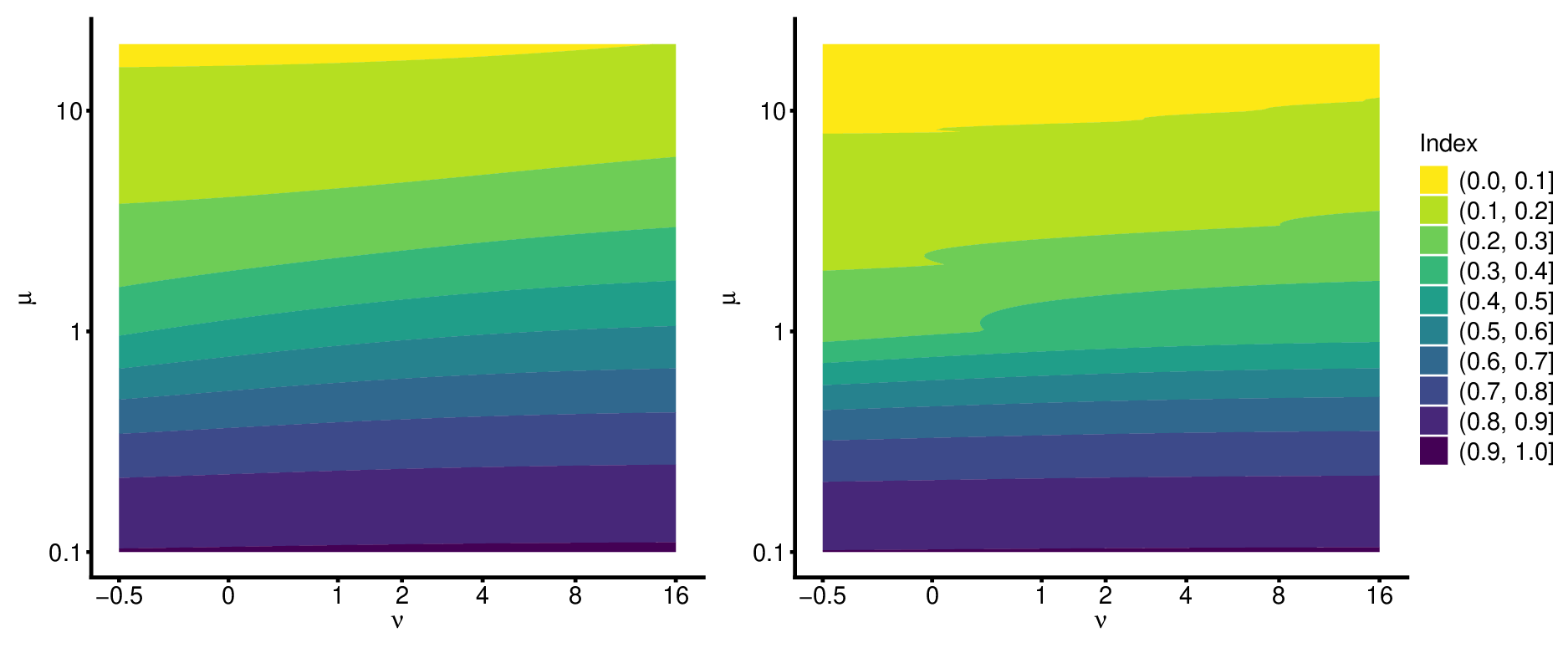}
    \caption{Contour plots of the Gini index (left) and Pietra index (right) for the Bessel Distribution as functions of $\mu$ and $\nu$. The vertical axis is displayed on the logarithmic scale and the horizontal axis is scaled as $\log(1+\nu)$.}
    \label{fig:Bessel_contours}
\end{figure}

\section{Concluding remarks}

We have shown that, for a broad class of density-dependent parasite death rates, the associated mean-parametrised parasite burden distributions are ordered in the convex order. This provides a Lorenz-order interpretation of the effect of density-dependent parasite mortality: stronger density dependence reduces aggregation when mean burden is fixed. The result applies to an idealised model in which infections occur singly, parasites do not reproduce within the host, and host mortality is independent of parasite burden. Extending the analysis to models incorporating these additional mechanisms is a natural topic for future work.

%\bigskip
%\noindent {\bf Acknowledgements:} The author express his thanks to the two reviewers for their thoughtful comments on the original version of the paper.

\bigskip
\noindent {\bf Funding information:} This research was supported by the UQ Winter Research Program.

\bigskip
\noindent {\bf Competing interests:} There were no competing interests to declare which arose during the creation of this article.

\bigskip
\noindent {\bf Declaration of generative AI and AI-assisted technologies in the manuscript preparation process:} During the preparation of this work the authors used OpenAI (ChatGPT-5.5) and Microsoft Copilot (GPT-5 reasoning model) in order to debug and optimise code, improve the visualisation and format of figures, and revise the text. After using this tools, the authors reviewed and edited the content as needed and take full responsibility for the content of the published article.

\bigskip
\noindent {\bf Data availability:} The code used to generate the figures in this paper are publicly available on Zenodo \citep{CM:2026}.

\bibliographystyle{plainnat}
\bibliography{parasites}

\end{document}